\documentclass[a4paper, twoside,12pt]{article}
\usepackage{fancyhdr}
\usepackage{fnpos}
 \usepackage[english]{babel}        %% adapte le style article aux conventions francophones
\usepackage[T1]{fontenc}          %% permet d'utiliser les caractï¿½res accentuï¿½s
\usepackage{graphicx}             %% permet d'importer des graphiques au format .EPS (postscript)
\usepackage{makeidx}
\usepackage{fancybox}
\usepackage{framed}
\usepackage{fancyhdr}
 \usepackage{pstricks,pst-plot,pstricks-add}
\usepackage[margin=1in]{geometry}
\usepackage{graphicx}
\usepackage{titlesec}
\usepackage{amsmath}
\usepackage{amsfonts}   % math fonts
\usepackage{amssymb}    % extra math symbols
\usepackage{amsthm}
\usepackage{dsfont}
\usepackage{mathtools}
\usepackage{verbatim}   % \begin{comment} multi-line comments \end{comment}
\usepackage{color}
\usepackage{listings}
\usepackage{graphicx}
\usepackage{amsmath}
\usepackage{amsmath,amssymb,color,inputenc,euscript,graphicx,psfrag}

\newtheorem{thm}{Theorem}[section]
\newtheorem{cor}[thm]{Corollary}
\newtheorem{lem}[thm]{Lemma}

\numberwithin{equation}{section}

\renewcommand{\thefootnote}

\newcommand\co{\operatorname{co}}

 {\markboth{\sl B.  Amri  and A. Gasmi }{\sl Dunkl  kernel }

\author {B\'echir Amri and Abdessalem Gasmi}

\title{On the estimates of   the  Dunkl  Kernel  }

\date{}
 
\begin{document}
 \maketitle
\begin{center}
    Department of Mathematics\\
College of Sciences\\
Taibah University\\
P.O. Box 30002
Al Madinah Al Munawarah, Saudi Arabia\\
\textbf{ e-mail:} bechiramri69@gmail.com; \textbf{ e-mail:}gasmi1975@yahoo.fr
\end{center}
  \begin{abstract}
 In this paper,  we are interested in the
estimates of  the   Dunkl  Kernel on some special sets,   following    the  work of  M.F.E. de Jeu and  M.  R\"{o}sler    in  \cite{R3}.

 \end{abstract}
\section{ Introduction }
Dunkl (1990-1991) originally defined    a family of differential-difference operators associated to
a finite reflection group on a Euclidean space, that eventually became associated with his name.
The eigenfunctions of the Dunkl operators, that known as Dunkl  kernel,   were first considered   by  Dunkl  \cite{D2}  and   have been later intensively studied   and investigated by several authors \cite{J1,D3,R3,R2}. One of the principal problem
that arises  in the  study of the Dunkl's kernels is   the  asymptotic behaviors of these functions,
which  were  known  only  for the reflection  group  $\mathbb{Z}_2^n$, and conjectured to have
an   extensions to all reflection groups. In this  work   we take up this  problem, we  obtain sharp  estimates when we restricted us to  cones   lie   in the interior  of the Weyl chamber.
   \par Let us  begin with a few definitions  and results as preliminary material. General references are \cite{D1,D2,R3,R2,R4,Xu}
\par Let $G\!\subset\!\text{O}(\mathbb{R}^n)$
be a finite reflection group associated to a reduced root system $R$
and $k:R\rightarrow[0,+\infty)$  be a $G$--invariant function
(called multiplicity function).
Let $R^+$ be a positive root subsystem. The Dunkl operators  $D_\xi^k$ on $\mathbb{R}^n$ are
the following $k$--de\-for\-ma\-tions of directional derivatives $\partial_\xi$
by difference operators\,:
\begin{equation}\label{Dxi}
D_\xi^k f(x)=\partial_\xi f(x)
+\sum_{\,\upsilon\in R^+}\!k(\upsilon)\,\langle\upsilon,\xi\rangle\,
\frac{f(x)-f(\sigma_\upsilon.\,x)}{\langle\upsilon,\,x\rangle}\,,
\end{equation}
where here
 $\sigma_\upsilon$ is the reflection
with respect to the hyperplane orthogonal to $ \upsilon$ and $\langle .,\, .\rangle$ is the usual Euclidean
inner product with  $| \,.\,  |$ its induced norm.
The operators $\partial_\xi$ and $D_\xi^k$
are intertwined by a Laplace--type operator
\begin{eqnarray*}\label{vk}
V_k\hspace{-.25mm}f(x)\,
=\int_{\mathbb{R}^n}\hspace{-1mm}f(y)\,d\nu_x(y)
\end{eqnarray*}
associated to a family of compactly supported probability measures
\,$\{\,\nu_x\,|\,x\!\in\!\mathbb{R}^n\hspace{.25mm}\}$\,.
Specifically, \,$\nu_x$ is supported in   the convex hull $\co(G.x)\,.$
\par For every $y\!\in\!\mathbb{C}^n$\!,
the simultaneous eigenfunction problem
\begin{equation*}
D_\xi^k f=\langle y,\xi\rangle\,f,
\qquad\forall\;\xi\!\in\!\mathbb{R}^n,
\end{equation*}
has a unique solution $f(x)\!=\!E_k(x,y)$
such that $E_k(0,y)\!=\!1$, called the Dunkl kernel and is given by
\begin{equation}\label{ros}
E_k(x,y)\,
=\,V_k(e^{\,\langle.,\,y\,\rangle})(x)\,
=\int_{\mathbb{R}^n}\hspace{-1mm}e^{\,\langle z,y\rangle}\,d\nu_x(z)
\qquad\forall\;x\!\in\!\mathbb{R}^n.
\end{equation}
 When $ k =0$ the Dunkl kernel $E_k(x,y)$ reduces to the exponential $e^{\langle x,y\rangle}$.Furthermore this kernel has a holomorphic extension to $\mathbb{C}^d\times \mathbb{C}^d $
 and the following estimate hold\,: for  $ \;x, \;y\!\in\!\mathbb{C}^d,$
\begin{itemize}
\item[(i)] $E_k(x,y)=E_k(y,x)$,
\item[(ii)] $E_k(\lambda x,y)=E_k(x,\lambda  y)$, for $\lambda\in \mathbb{C}$
\item[(iii)] $E_k(g. x,g.y)=E_k(x, y)$, for $g\in G$.
\item[(iv)] If we design by  $x^+$   the intersection point of any
  orbit $G.x$ with the closure of the weyl chamber $\overline{C}$, then  for $z\in \mathbb{C}$
\begin{equation}\label{Ez}
 |E_k(zx, y)|\leq e^{Re(z)\langle x^+, \;y^+\rangle}.
\end{equation}
  \end{itemize}
 In dimension $d\!=\!1$,
these functions can be expressed in terms of hypergeometric function $_1F_1$, specially
$$E_k(x,y)=E_k(xy)= e^{xy} \;_1F_1(2k,2k+1,-2xy)$$
  and from the behavior of  $_1F_1$ (see, e.g. \cite{MA})  one can deduce the estimates
 \begin{equation}\label{d=1}
  | E_k(xy)|\leq c\frac{e^{|xy|}}{|xy|^{k}};\qquad  | E_k(ixy)|\leq c|xy|^{-k};\quad x,y\in \mathbb{R}\setminus\{0\}.
 \end{equation}
 The subject is  then    a generalisation of these estimates to any reflection group.
 Then (\ref{d=1}) become
 \begin{equation}\label{d=n}
  | E_k(x,y)|\leq c\frac{e^{ \langle x^+,y^+\rangle}}{ \sqrt{w_k(x)w_k(y)}};\qquad  | E_k(ix,y)|\leq \frac{c}{ \sqrt{w_k(x)w_k(y)}}  ;\quad x,y\in \mathbb{R}^n\setminus \bigcup_{\alpha\in R^+} H_\alpha.
 \end{equation}
 where
$$ w_k(x)=\,\prod_{\,\upsilon\in R^+}|\,\langle\upsilon,x\rangle\,|^{\,2\,k(\upsilon)}$$
which is $G$-invariant and homogeneous of degree $2\gamma_k$,
$$\gamma_k=\sum_{\alpha\in R^+}k(\alpha).$$
In this paper we shall prove that (\ref{d=n}) hold for every $x,y\in C_\delta$, $\delta>0$ where
$$ C_\delta=\{x\in C;\; \langle x,\alpha \rangle\geq \delta|x| \}$$
with constant $c$ depends only on $\delta$.
 The motivation of studying the asymptotic behavior at infinity  arises from the work of De Jeu and R\"{o}sler in  \cite{R3} where they proved    the following result
\begin{thm}
 There exists a constant non-zero vector $v=(v_g)_{g\in G}$
such that for all  $y\in C$ and $g\in G$
$$\lim_{|x|\rightarrow\infty\; x\in C_\delta}\sqrt{w_k(x)w_k(y)}e^{\langle-ix,gy\rangle} E_k(ix,gy)=v_g.$$
\end{thm}
\section{ The main estimates for the Dunkl kernel}
 In this work we  may assume that  $\gamma_k > 0$ and the root system $R$ engender the space $\mathbb{R}^n$.  Let $\Delta$ be the  basis of $\mathbb{R}^n$ consists of    the simple roots of $R$.
 Recall that  every root system   has a set of simple root  such that for each   root may be written as  a linear combination of simple roots  with coefficients all of  the same sign,  ( see, e.g.  \cite{Hum}). Consider $(\lambda_i)_{1\leq i\leq n}$ the dual basis  of  $\Delta$. Then the  fundamental Weyl chamber is given by
$$C=\{x\in \mathbb{R}^n,\; x=\sum_{i=1}^nx_i\lambda_i,\; \lambda_i>0, \forall i=0,...,n\}.$$
 \par Let    $(v_i)_{1\leq i\leq n}$ be   a family of    linearly independent vectors and   $\Lambda$ be the convex polytope  defined by
$$\Lambda_{v_1,...,v_n}= \{x\in \mathbb{R}^n;\;x =\sum_{i=1}^nx_i\nu_i,\;\; x_i>0\}.$$
\begin{lem}\label{lll}
  For all $\delta>0$ there exists    a family  of    linearly independent vectors  $(v_i)_{1\leq i\leq n}$ ,  such that
$$C_\delta\subset \Lambda_{v_1,...,v_n}$$
\end{lem}
\begin{proof}
  Let $\Pi_\delta$  be the   set,
$$\Pi_\delta=\{x\in \mathbb{R}^n;\;|x|=1,\;  \langle x,\alpha \rangle\geq \delta|x| \}$$
and put $\lambda=\sum_{i=1}^n\lambda_i$.
For all integer   $p\geq 1$ define the vectors
$$v_{p,i}=\lambda_i+\frac{\lambda}{p}, \quad i=1,...,n.$$
It is easy to see that the vectors $v_{p,i}$  are  linearly independent and
\begin{equation}\label{vi}
v_{p,i}=  v_{p+1,i}+\frac{1}{p(2p+1)} \sum_{j=1}^nv_{p+1,j}.
\end{equation}
for all $i=1,...,n$ and  $p\geq 1$. Denote  $\Lambda^p=\Lambda_{v_{p,1},...,v_{p,n}}$. It follows from (\ref{vi})  that
$$\Lambda^{p}\subset\Lambda^{p+1}$$
Next  we claim that
$$ \bigcup_p \Lambda^p=C.$$
 Clearly $ \bigcup_p \Lambda^p\subset C$. If  $x\in C$, $x=\sum_{i=1}^nx_i\lambda_i$ with $x_i>0$ then there exists $p \geq 1$
such that
$$x_i-\frac{\sum_{i=1}^nx_i}{p}>0$$
and so
$$x= \sum_{i=1}^n\left(x_i-\frac{\sum_{i=1}^nx_i}{p}\right) v_{p,i}\in \Lambda^p.$$
Now since
$$ \Pi_\delta\subset \bigcup_p \Lambda^p $$
and  $\Pi_\delta$ is compact then there exists $p_0$ such that $\Pi_\delta \subset \Lambda^{p_0}$ and by convexity
$C_\delta \subset \Lambda^{p_0} $.  This conclude the proof of Lemma \ref{lll}.
\end{proof}
\begin{lem}\label{l1}
   For all $x,y\in C$ and $g\in G$ the function
$t\rightarrow t^{\gamma_k}e^{-t\langle x,y\rangle}E_k(tx,gy)$, $t\geq0$ is bounded.
 \end{lem}
\begin{proof}
The lemma follows using  the Phragm\'{e}n-Lindel\"{o}f Theorem (see, e.g. \cite[section 5.61]{Tit}) by considering the  functions of a complex variable
 $$g(z)=z^{\gamma_k}e^{-z\langle x,y\rangle}E_k(zx,gy), \quad z\in H = \{z \in  \mathbb{C},\;  Re(z) \geq 0\}.$$
\par Indeed, by (\ref{Ez}) we have $$|g(z)|\leq |z|^{\gamma_k}=O(e^{|z|^\delta}),\;  |z|\rightarrow\infty, \;\; \forall \; \delta>0 $$
and    from \cite[Corollary 1]{R3},   the function   $t\rightarrow   t^{\gamma_k}e^{-\pm it\langle x,gy\rangle}E_k(\pm itx,y)$ is bounded.
\end{proof}
\begin{lem}\label{l2}
  Given a polytope convex  $\Lambda_{v_1,...,v_n}$ there exists a constant $c>0$ such that for  all   $g\in G$ and $x,y\in \Lambda_{v_1,...,v_n}$,  $x=\sum_{i=1}^nx_iv_i$, $y=\sum_{i=1}^ny_iv_i$.
 \begin{equation}\label{eq}
   0\leq E_k(x,g.y)\leq c \frac{e^{\langle x ,\;y\rangle}}{\prod_{i,j=1}^{n}(x_iy_i)^{\gamma_k/n} }
 \end{equation}
   \end{lem}
 \begin{proof} By using  Holder's inequality   in the integral formula  (\ref{ros}) it follows that
  \begin{eqnarray*}
    E_k(x,y)&\leq& \prod_{i=1}^{n}  E_k(nx_iv_i,y)\Big)^{1/n}
  \\&\leq&  \prod_{i=1}^{n}\prod_{j=1}^{n} E_k(n^2x_iy_jv_i,gv_j)\Big)^{1/n^2}
  \\&\leq&\prod_{i=1}^{n}\prod_{j=1}^{n}\Big((n^2(x_iy_j))^{\gamma_k}e^{-n^2 x_i y_j\langle  v_i,gv_j\rangle} E_k(n^2x_iy_jv_i,gv_j)\Big)^{1/n^2}  \frac{e^{\langle x,\;y\rangle}}{n^2\prod_{i,j=1}^{n}|x_iy_j|^{\gamma_k/n^2} }.
    \end{eqnarray*}
Then conclude Lemma \ref{l2} from  Lemma \ref{l1}.
 \end{proof}
Next  for $x= \sum_{i=1}^nx_iv_i$ we put
$$x^*= \sum_{i=1}^n|x_i|v_i.$$
 Proceeding as in the above lemma one can state
\begin{lem}
If $G$ contains  $-Id_{\mathbb{R}^n}$ then  the dunkl kernel satisfies the following estimate
 \begin{equation*}
   0\leq E_k(x,y)\leq c\; \frac{e^{\langle x^*,\;y^*\rangle}}{\prod_{i,j=1}^{n}|x_iy_i|^{\gamma_k/n} }
 \end{equation*}
   For all $x,y\in \mathbb{R}^n$
 \end{lem}
Now the main estimate states as follows
\begin{thm} Given a polytope convex $ \Lambda_{v_1,...,v_n}$,  there exists a constant  $c>0$, depending only on the choice of the vectors $v_i$  such that
\begin{equation}\label{eq}
   E_k(x,gy)\leq C \;\frac{e^{\langle x,\;y\rangle}}{\sqrt{ w_k(x)w_k(y)}},
 \end{equation}
   for all $x,y\in   \Lambda_{v_1,...,v_n}$  and $g\in G$
 \end{thm}
 \begin{proof}
Choose a vectors $\xi_1,...,\xi_n \in C$, linearly independent such that
$$\Lambda_{v_1,.,.,.,v_n}\subset \Lambda_{\xi_1.,.,.,\xi_n}$$
Consider $\xi_1,...,\xi_n$ as a basis of $\mathbb{R}^n$ and let the sets
$$H_{_{p,i}}=\{x\in \Lambda_{\xi_1.,.,.,\xi_n},\; \; x=\sum_{i=1}^n x_i\xi_i,\;\; x_1/p\leq x_i\leq px_1\},\quad p\in \mathds{N}^*.$$
Clearly $H_{p,i}\uparrow \Lambda_{\xi_1.,.,.,\xi_n}$ as $p\rightarrow +\infty$. Thus one can find $p_i$ such that
\begin{equation}\label{lam}
  \Lambda_{v_1,.,.,.,v_n}\subset H_{p_i,i}.
\end{equation}
For  $x= \sum_{i=1}^n x_i\xi_i\in \Lambda_{v_1,.,.,.,v_n}$ and $\alpha\in R^+$, with the use  of (\ref{lam})   we have
\begin{eqnarray*}
 \langle x,\alpha\rangle &=& \sum_{i=1}^n x_i\langle \xi_i,\alpha\rangle
\leq   \max_{i}\langle \xi_i,\alpha\rangle\sum_{i=1}^n x_i \leq c \left(\prod_{i=1}^n x_i\right)^{1/n}
\end{eqnarray*}
Now from Lemma \ref{l2}
 \begin{eqnarray*}
  E_k(x,g.y)&\leq& c \frac{e^{\langle x ,\;y\rangle}}{\prod_{i,j=1}^{n}(x_iy_i)^{\gamma_k/n} }
\\ &\leq &c \frac{e^{\langle x ,\;y\rangle}}{\prod_{\alpha\in R^+}\left(\prod_{i=1}^{n}x_i \right)^{k_\alpha/n} \left(\prod_{i=1}^{n}y_i \right)^{k_\alpha/n}  }\\
&\leq &  c \;\frac{e^{\langle x,\;y\rangle}}{\sqrt{ w_k(x)w_k(y)}},
\end{eqnarray*}
 which is the desired estimate.
\end{proof}
We came  now to the second part of this work, that is the behavior of the kernel $E_k(ix,y)$.  The   main
result is the following
\begin{thm}\label{tth}
For $g\in G$ there exists a constant non-zero vector  $v_g$
such that  for each $\delta>0$
 \begin{equation}\label{eq}
 \lim_{|x|\times|y|\rightarrow+\infty;\;x,y\in C_\delta}  \sqrt{ w_k(x)w_k(y)}e^{-i\langle x,\;gy\rangle}E_k(ix,gy)=v_g.
 \end{equation}
    \end{thm}
\begin{proof}
  We proceed as in the proof of  \cite[Theorem 1]{R3}. Keeping the notations of \cite{R3} we consider  the  function
$$F_g(x,y)= \sqrt{ w_k(x)w_k(y)}e^{-i\langle x,\;gy\rangle}E_k(ix,gy); \quad  (x,y)\in \mathbb{R}^n\times \mathbb{R}^n$$
and $F=(F_g)_{g\in G}$.
According to \cite[Lemma 1]{R3}, if $\xi=(\xi_1,\xi_2)\in \mathbb{R}^n\times \mathbb{R}^n$, then
\begin{eqnarray*}
  \partial_\xi F_g(x,y)=\sum_{\alpha\in R^+} k(\alpha)\left(\frac{\langle \alpha,\xi_1\rangle }{\langle \alpha,x\rangle}+\frac{\langle \alpha,g\xi_2\rangle }{\langle \alpha,gy\rangle}\right)
e^{-i\langle \alpha,x\rangle\langle \alpha, gy\rangle}F_{\sigma_\alpha g}(x,y).
 \end{eqnarray*}
 Let $\delta>0$ and $\kappa=(\kappa_1,\kappa_2)$ be  a curve of $\mathbb{R}^n\times \mathbb{R}^n $ such that    $\kappa_1,\kappa_2:\;(0,+\infty)\rightarrow  C_\delta$ are   two admissible curves in the sense given in \cite{R3}. Define
  $F^\kappa(t)=(F^\kappa_g)_{g\in G}$ where $F^\kappa_g(t)=F_g(\kappa_1(t),\kappa_2(t))$.
We have
\begin{eqnarray*}
 (F^\kappa_g)'(t)=\sum_{\alpha\in R^+} k(\alpha)\left(\frac{\langle \alpha,\kappa'_1(t)\rangle }{\langle \alpha,\kappa_1(t)\rangle}+\frac{\langle \alpha,g\kappa'_2(t)\rangle }{\langle \alpha,g\kappa_2(t)\rangle}\right)
e^{-i\langle \alpha,\kappa_1(t)\rangle\langle \alpha, g\kappa_2(t)\rangle}F^\kappa_{\sigma_\alpha g}(t)
 \end{eqnarray*}
and $F^{\kappa}$ satisfies the differential equation
$$(F^{\kappa})'(t)=A^\kappa(t)F^{\kappa} (t)$$
where the matrix $A^\kappa(t)$ is given by $A(t)= \sum_{\alpha\in R^+}k(\alpha ) B^\kappa_\alpha(t)$ and
$$B_{g,h}(t)=\left\{
  \begin{array}{ll}
   \left(\dfrac{\langle \alpha,\kappa'_1(t)\rangle }{\langle \alpha,\kappa_1(t)\rangle}+\dfrac{\langle \alpha,g\kappa'_2(t)\rangle }{\langle \alpha,g\kappa_2(t)\rangle}\right)e^{-i\langle \alpha,\kappa_1(t)\rangle\langle \alpha, g\kappa_2(t)\rangle}\;\;\text{if}\; & \hbox{ $h=\sigma_\alpha g$.}\\ 0&\;\;\; \hbox{  \text{otherwise.}}$$
  \end{array}
\right.$$
We will try to apply the proposition 1 of \cite{R3}. For arbitrary $ t>0$
\begin{eqnarray*}
 \int_t^{+\infty} B_{g,\sigma_\alpha g}(s)ds&=&\int_t^{+\infty}\left(\dfrac{\langle \alpha,\kappa'_1(t)\rangle }{\langle \alpha,\kappa_1(t)\rangle}+\dfrac{\langle \alpha,g\kappa'_2(t)\rangle }{\langle \alpha,g\kappa_2(t)\rangle}\right)e^{-i\langle \alpha,\kappa_1(t)\rangle\langle \alpha, g\kappa_2(t)\rangle}\\
&=&\int_{\left|\langle \alpha,\kappa_1(t)\rangle\langle \alpha, g\kappa_2(t)\rangle\right|}^{+\infty}\frac{e^{-i\;sign(g^{-1}\alpha)\;u}}{u}\;du
\end{eqnarray*}
where for $\beta\in R$
$$ sign( \beta)=\left\{
     \begin{array}{ll}
       1, & \hbox{ if $\beta\in R^+$}\\
 - 1, & \hbox{ if $\beta\in R^-.$}\\
     \end{array}
   \right.$$
  This integral exists. Next  we are led to examine the integrability of
\begin{eqnarray*}
  &&I_{\alpha,\beta,g}(t)\\&&=\left(\dfrac{\langle \alpha,\kappa'_1(t)\rangle }{\langle \alpha,\kappa_1(t)\rangle}+\dfrac{\langle \alpha,g\kappa'_2(t)\rangle }{\langle \alpha,g\kappa_2(t)\rangle}\right)e^{-i\langle \alpha,\kappa_1(t)\rangle\langle \alpha, gk_2(t)\rangle}\int_{\left|\langle \beta,\kappa_1(t)\rangle\langle\beta, \sigma_\alpha g\kappa_2(t)\rangle\right|}^{+\infty}\frac{e^{-i\;sign(g^{-1}\alpha)\;u}}{u}\;du
\end{eqnarray*}
with $g\in G$ and $\alpha,\beta\in R^+$. Observe that
\begin{eqnarray*}
\left|\int_{\left|\langle \beta,\kappa_1(t)\rangle\langle\beta, \sigma_\alpha g\kappa_2(t)\rangle\right|}^{+\infty}
\frac{e^{-i\;sign(g^{-1}\alpha)\;u}}{u}\;du\right|&\leq &\frac{2}{\left|\langle \beta,\kappa_1(t)\rangle\langle\beta, \sigma_\alpha g\kappa_2(t)\rangle\right|}
\\ &\leq & \frac{c}{\left|\langle\alpha,\kappa_1(t)\rangle\langle\alpha,  g\kappa_2(t)\rangle\right|},
\end{eqnarray*}
since $|\langle\beta, \sigma_\alpha g\kappa_2(t)\rangle|=|\langle g^{-1}\sigma_\alpha\beta,  \kappa_2(t)\rangle|\geq \delta|\kappa_2(t)|\geq \delta\;|\langle\alpha,  g\kappa_2(t)\rangle|\sqrt{2}$ and similarly  $|\langle\beta,  \kappa_1(t)\rangle|= \delta\;|\langle\alpha,   \kappa_1(t)\rangle|\sqrt{2}$. Then
$$|I_{\alpha,\beta,g}(t)|\leq c\; sign(g^{-1}\alpha) \;\dfrac{\langle \alpha,\kappa'_1(t)\rangle \langle \alpha,g\kappa_2(t)\rangle+\langle \alpha,g\kappa'_2(t)\rangle\langle \alpha,\kappa_1(t)\rangle  }{(\langle \alpha,\kappa_1(t)\rangle\langle \alpha,g\kappa_2(t)\rangle)^2} $$
 and for $t_0>0$
$$\int_{t_0}^{+\infty}|I_{\alpha,\beta,g}(t)| dt \leq \frac{c}{\langle\alpha,\kappa_1(t_0)\rangle|\langle\alpha,  g\kappa_2(t_0)\rangle|}\; .$$
Applying now the proposition of \cite{R3}, we conclude that
$$\lim_{t\rightarrow+\infty}F^\kappa_g(t)$$
 exists and    different from zero .
It remains to justify that  is independent of the choice of the curves $\kappa_1$ and $\kappa_2$. Also this can be do by  the same argument  of
the proof of Theorem 1 of \cite{R3}. Given $\ell_1$ and $\ell_2$ be an other admissible curves in $C_{\delta}$. One can  construct an admissible sequences  $(x_n)_n$ and $(y_n)_n$   in $C_\delta$ such that $( x_{2n+1}, y_{2n+1})\in (\kappa_1,\kappa_2)$ and  $( x_{2n}, y_{2n})\in (\ell_1,\ell_2)$. Let $r_1$ and $r_2$ be  an
interpolating  curves respectively of $(x_n)_n$ and $(y_n)_n$. Thus we may have
$$\lim_{t\rightarrow \infty}F_g(\kappa_1(t),\kappa_2(t))=\lim_{t\rightarrow \infty}F_g(\ell_1(t),\ell_2(t))
=\lim_{t\rightarrow \infty}F_g(r_1(t),r_2(t)).$$
In particular if  we choose  $\kappa_1(t)=t x$  and $\kappa_2(t)=t y $ for   fixed vectors $x$ and $y$ then  from  Theorem 1 of \cite{R3} there exists a constant non-zero vector $v=(v_g)_{g\in G}\in \mathbb{G}^{|G|}$  such that
 $$\lim_{t\rightarrow \infty}F_g(\kappa_1(t),\kappa_2(t))=\lim_{t\rightarrow \infty}F_g(t^2x ,y)=v_g.$$
Let us now show that
$$\lim_{|x|\times|y|\rightarrow+\infty;\;x,y\in C_\delta}  \sqrt{ w_k(x)w_k(y)}e^{-\langle ix,\;gy\rangle}E_k(ix,gy)=v_g
$$
 Indeed, if it is not , then  we can find $\varepsilon>0$ and a sequences $(x_n)_n$ and $(y_n)_n$ of $C_\delta$ such that
$|x_n|\times|y_n|\rightarrow+\infty$ and
$$|F_g(x_n,y_n)-v_g|>\varepsilon$$
we  can assume that $|x_n| \rightarrow+\infty$ and $ |y_n|\rightarrow+\infty$, since we have
$$F_g(x_n,y_n)=F_g(\sqrt{|x_n||y_n|}x_n',(\sqrt{|x_n||y_n|}y_n')$$
where $x'_n=x_n/|x_n|$ and $y_n=y'_n/|y_n|$. We may also assume   that $(x_n)_n$  and $(y_n)_n$ are  admissibles.
  Thus it  yield that
$$\lim_{t\rightarrow \infty}F_g(x_n ,y_n)=v_g$$
which is a contradiction.
\end{proof}
\begin{cor}
 Let $\delta>0$. There exists a constant  $c>0$,   such that
\begin{equation}\label{EE}
   |E_k(ix,gy)|\leq   \;\frac{c}{\sqrt{ w_k(x)w_k(y)}}
 \end{equation}
   for all $x,y\in C_\delta$  and $g\in G$.
\end{cor}
\begin{proof}
From Theorem \ref{tth} we can find $M>0$ such that for all $x,y\in C_\delta$, $|x|\times|y|\geq M$ we have
$$\sqrt{ w_k(x)w_k(y)} |E_k(ix,gy)|\leq 2|v_g|$$
for all $g\in G$. When  $|x|\times|y|\leq M$,
by use of  (\ref{Ez}), one obtains
$$\sqrt{ w_k(x)w_k(y)} |E_k(ix,gy)|\leq \left(\prod_{\alpha\in R^+}|\alpha|\right)^{2\gamma_k}\;(|x|\times|y|)^\gamma_k\leq
 M^\gamma_k\;\left(\prod_{\alpha\in R^+}|\alpha|\right)^{2\gamma_k}. $$
Hence (\ref{EE}) follows.
\end{proof}


\begin{thebibliography}{1}
\bibitem{MA}
 M. Abramowitz and I.A. Stegun, \textit{Pocketbook of Mathematical Functions},
Verlag Harri Deutsch, Frankfurt/Main, 1984.
 \bibitem{J1}
M.F.E. de Jeu,
\textit{The Dunkl transform\/},
Invent. Math. 113 (1993), no. 1, 147-162.
\bibitem{R3}
  M.F.E. de Jeu and M.  R\"{o}sler, \textit{Asymptotic Analysis for the Dunkl Kernel}, J. Approx. Theory 119
(2002), no. 1, 110-126.
 \bibitem{D1}
C. F. Dunkl,
\textit{Differential--Difference operators associated to reflextion groups\/},
Trans. Amer. Math. 311 (1989), no. 1, 167--183.
\bibitem{D3}
C. F. Dunkl, \textit{Hankel transforms associated to finite reflection groups}, Contemp. Math., vol. 138, 1992, pp. 123-138.

 \bibitem{D2}
  C.F. Dunkl,  \textit{Integral kernels with reflection group invariance}. Can J. Math. 43,
1213-1227 (1991).

 \bibitem{Hum}
James E. Humphreys,\textit{ Reflection Groups and Coxeter Groups}, Cambridge University Press, 1990
\bibitem{R2}
M.  R\"{o}sler, \textit{Positivity of Dunkl's intertwining operator}, Duke Math. J. 98 (1999) 445-463.
\bibitem{R4}
M. R\"osler,
\textit{Dunkl operators\,: theory and applications\/},
in \textit{Orthogonal polynomials and special functions
(Leuven, 2002)\/},
Lect. Notes Math. 1817, Springer--Verlag (2003), 93--135.
 \bibitem{Tit}
E. Ch. Titchmarsh,  \textit{The Theory of Functions (Second ed.)}. Oxford University Press. ISBN 0-19-853349-7.
\bibitem{Xu}  S. Thangavelu and Y. Xu,\textit{ Convolution operator and maximal function for the Dunkl transform}, J. Anal. Math. 97
(2005), pp. 25-56.
\end{thebibliography}
  \end{document}